\documentclass[11pt]{amsart}
\usepackage{amssymb, amsmath, amsmath}
\usepackage[backref]{hyperref}
\usepackage{mathrsfs}
\usepackage{amscd}   
\usepackage{fullpage}
\usepackage[all]{xy} 
\usepackage{setspace}
\usepackage[margin=1.86  cm, headsep = 0.1 cm, footskip = 0.7 cm]{geometry}
\usepackage[english]{babel}
\usepackage{amsfonts}
\usepackage{xurl}
\usepackage[T1]{fontenc}
\usepackage{comment}
\usepackage{enumitem}
\usepackage{tikz-cd}
\usepackage{yfonts}

\newtheorem{lemma}{Lemma}[section]
\newtheorem{theorem}[lemma]{Theorem}
\newtheorem{proposition}[lemma]{Proposition}

\newtheorem{claim*}{Claim}

\newtheorem*{thm*}{Theorem}

\theoremstyle{definition}
\newtheorem{defn}[lemma]{Definition}
\newtheorem{remark}[lemma]{Remark}

\newtheorem{example}[lemma]{Example}

\numberwithin{equation}{section}

\newcommand{\bc}{\boldsymbol{c}}
\newcommand{\bd}{\boldsymbol{d}}
\newcommand{\bh}{\boldsymbol{h}}
\newcommand{\bs}{\boldsymbol{s}}
\newcommand{\bt}{\boldsymbol{t}}
\newcommand{\bu}{\boldsymbol{u}}
\newcommand{\bv}{\boldsymbol{v}}
\newcommand{\bx}{\boldsymbol{x}}
\newcommand{\bgamma}{\boldsymbol{\gamma}}
\newcommand{\bzeta}{\boldsymbol{\zeta}}

\newcommand{\beps}{\boldsymbol{\varepsilon}}
\newcommand{\bone}{\boldsymbol{1}}
\newcommand{\bzero}{\boldsymbol{0}}

\newcommand{\R}{\mathbb{R}}
\newcommand{\Z}{\mathbb{Z}}
\newcommand{\N}{\mathbb{N}}
\newcommand{\Q}{\mathbb{Q}}
\newcommand{\PP}{\mathbb{P}}

\newcommand{\frakm}{\mathfrak{m}}

\newcommand{\Spec}{\mathrm{Spec}}

\def\calA{\mathcal{A}}
\def\calD{\mathcal{D}}
\def\calM{\mathcal{M}}
\def\calO{\mathcal{O}}
\def\calP{\mathcal{P}}
\def\calT{\mathcal{T}}
\def\calX{\mathcal{X}}

\DeclareMathOperator{\val}{val}
\def\mmod#1{\;(\mathrm{mod}\;{#1})}

\begin{document}

\setcounter{tocdepth}{2}

\title{Campana points on diagonal hypersurfaces}
  
\author{Francesca Balestrieri, Julia Brandes, Miriam Kaesberg, Judith Ortmann, Marta Pieropan, Rosa Winter} 

\address{The American University of Paris, 5 Boulevard de la Tour-Maubourg, 75007, Paris, France}
\email{fbalestrieri@aup.edu}
\address{Mathematical Sciences, University of Gothenburg and Chal\-mers Institute of Technology, 412 96 G\"oteborg, Sweden}
\email{brjulia@chalmers.se}
\address{University of G\"ottingen, Mathematical Institute, Bunsenstra{\ss}e 3-5, 37073 G\"ottingen, Germany}
\email{miriam.kaesberg@mathematik.uni-goettingen.de}
\address{Institut f\"ur Algebra, Zahlentheorie und Diskrete Mathematik, Leibniz Universit\"at Hannover, Welfengarten 1, 30167 Hannover, Germany}
\email{ortmann@math.uni-hannover.de}
\address{Utrecht University, Mathematical Institute, Budapestlaan 6, 3584 CD Utrecht, the Netherlands}
\email{m.pieropan@uu.nl}
\address{King's College London, Department of Mathematics, Strand, London, WC2R 2LS, United Kingdom}
\email{rosa.winter@kcl.ac.uk}

\date{\today}


\maketitle

\begin{abstract}
We construct an integral model for counting Campana points of bounded height on diagonal hypersurfaces of degree greater than one, and give an asymptotic formula for their number, generalising work by Browning and Yamagishi. The paper also includes background material on the theory of Campana points on hyperplanes and previous results in the field.
\end{abstract}

\setcounter{tocdepth}{1}
\tableofcontents

\section{Introduction}
The study of \textsl{Campana points} on varieties has gained a lot of attention in recent years. Philosophically, sets of Campana points on a variety over a number field interpolate between the set of rational points and the set of integral points. 
As such, the general aim when studying Campana points is to use techniques from the study of rational points to say something about Campana points, which can in turn be used to understand integral points better. There are several definitions of Campana points (see Section \ref{sec: Campana points}), all having in common that they are integral points on a proper variety satisfying prescribed
`intersection conditions' with respect to a boundary divisor. When $X$ is a hypersurface in a projective space $\mathbb{P}_{\mathbb{Q}}^n$ 
and the closure of $X$ in $\mathbb{P}_{\mathbb{Z}}^n$ is a flat proper model of $X$ over $\mathbb Z$, a set of Campana points on $X$ for the boundary divisor induced by the coordinate hyperplanes of the projective space is the set of $\mathbb Z$-points in $X$ for which the coordinates are \textsl{$m$-full} for different $m$ (an integer $a$ is $m$-full if $p|a$ implies $p^m|a$ for all primes $p$), see Example \ref{example}.

Many classical results and conjectures for rational points have been formulated also for Campana points; an overview of some of the following can be found in \cite{Tanimoto21}. An analogue of the Mordell Conjecture for Campana points (\textsl{orbifold Mordell conjecture}) over function fields was introduced and proven in characteristic~0 in \cite{Ca05}, and in positive characteristic in \cite{KPS22}. 
Asymptotics for the number of Campana points of bounded height were first given in \cite{VV12} 
for linear hypersurfaces in $\mathbb P^n_{\Q}$ for $n\geq 4$, with the prescribed intersection conditions corresponding to points with squareful coordinates. 
The same was done for $n=3$ in \cite{Shute21}, while \cite{BY21} deals with general $m$-full conditions on each coordinate for $n$ large enough (see Section~\ref{subsec: applications circle method to campana points} for more details on these counting results). Upper and lower bounds for $n=2$ are given in \cite{BVV12}. 
In \cite{PSTVA21}, the authors gave an analogue of the Manin Conjecture for Campana points of bounded height on Fano orbifolds (\textsl{PSTVA Conjecture}), and proved it for equivariant compactifications of vector groups. 
The PSTVA Conjecture has also been proven for split toric varieties with torus invariant boundary \cite{PS, Santens22}, and biequivariant compactifications of the Heisenberg group \cite{Xiao22}. It is also compatible with the order of growth in the results in \cite{VV12}, \cite{BVV12}, \cite{BY21} and \cite{Shute21}. 
However, in \cite{Shute22} it is shown that the conjectured leading constant in \cite{BVV12} does not agree with the one predicted by the PSTVA Conjecture for the orbifold $(\mathbb P^1,\frac 12 [0] +\frac 12 [\infty] +\frac 12 [1])$, and a counterexample to the leading constant predicted by the PSTVA Conjecture is given by counting Campana points on $\mathbb{P}^1$ with the boundary divisor $\frac 12\{37x^2+109y^2=0\}$. For the same orbifolds, \cite{Shute22} shows that the leading constant can be made arbitrarily small by excluding suitable thin sets from the counting problem. To the best of our knowledge, it has not yet been investigated whether the leading constants in \cite{VV12}, \cite{BY21} and \cite{Shute21} agree with the PSTVA Conjecture.

The first asymptotics for the arguably harder case of \textsl{weak} Campana points are given in \cite{Streeter22}, where the author counts such points of bounded height on orbifolds associated to norm forms for Galois extensions of number fields. 
The same paper also contains the first asymptotics for Campana points on \textsl{singular} orbifolds. Weak Campana points also appear in \cite{Santens22}, where the author counts diagonal quartic surfaces with a Brauer-Manin obstruction to the Hasse principle, and shows that this number gives a lower bound for the number of weak Campana points in $\mathbb{P}^3$ for the boundary divisor induced by the coordinate hyperplanes. The study of the potential density of rational points (i.e., density of rational points over a finite extension of the base field) 
was translated to Campana points as well, with Campana first stating a conjecture in this direction for Campana points on curves over number fields in \cite{Ca05}. 
In \cite{AVA18}, the authors extended the conjecture to a much larger class of varieties, and they showed that their conjecture recovers Lang's conjecture for rational points on varieties of general type, and the Lang-Vojta conjecture for $S$-integral points on varieties of logarithmic general type. The authors of \cite{BD19} proved the conjecture when the boundary divisor is a general smooth divisor satisfying some extra conditions. Instances of the conjecture formulated over function fields are proven in \cite{RTW21} and in \cite{GF22}. 
Finally, analogues of local-global principles for Campana points are also studied: in \cite{NS20} the authors defined the Campana analogues of weak weak approximation and the Hilbert property, and proved the Campana version of a result by Colliot-Thélène and Ekedahl that weak weak approximation implies the Hilbert property. A Campana analogue of the Hasse Principle is defined in \cite{MNS22}, where the authors studied weak approximation and Brauer-Manin obstructions for Campana points.

\vspace{11pt}

In the setting mentioned earlier, where $X$ is a hypersurface in projective space over $\mathbb{Q}$, counting Campana points of bounded height is equivalent to counting $m$-full solutions of bounded size to homogeneous equations. This counting problem lends itself very well to the \textsl{circle method}, a powerful counting method in analytic number theory. So far, the circle method has been used to count Campana points when both $X$ and the boundary divisor are linear \cite{VV12,BVV12,BY21,Shute21}; we give more detail on this in Section \ref{subsec: applications circle method to campana points}. In this paper we extend the techniques in \cite{BY21} to study the problem of counting Campana points on a \textsl{non-linear} diagonal hypersurface $X$ in $\mathbb{P}^n$. Our main result is the following.

\begin{theorem}\label{thm: diagonal hypersurface}
Let $k,n,m_0,\ldots, m_n$ be positive integers. Let $B$ be a positive real number.
Let $X\subset\mathbb{P}^n_{\mathbb{Q}}$ be the hypersurface given by $\sum_{i=0}^nc_ix_i^k=0$ with $c_0,\dots,c_n\in\mathbb{Z}_{\neq0}$ and $\gcd(c_0,\dots,c_n)=1$. Let $D$ be the $\mathbb{Q}$-divisor on $X$ given by $D=\sum_{i=0}^n(1-\frac 1{m_i})\{x_i=0\}$.
With the model defined in Section \ref{sec: proper model} and the height function induced by the Weil height on $\mathbb P^n(\mathbb Q)$, the set of Campana $\mathbb Z$-points of height at most $B$ on $X\cap \left(\mathbb{P}^n\smallsetminus \bigcup_{i=0}^n \{x_i=0\} \right)$ with respect to the boundary divisor $D$ is given by \[
   N(X,D,B) =  \left\{ (x_0: \dots: x_n) \in \mathbb P^n(\mathbb Q) \left| \begin{array}{c} 
   x_0,\dots,x_n\in \mathbb Z_{\neq 0},\\
   \gcd(x_0, \dots, x_n)=1,\\    
    x_i \textup{ is } m_i \textup{-full}\ \forall i\in\{0,\dots,n\},\\ 
    |x_0|, \dots,|x_n| \le B, \\
    c_0x_0^k+\dots +c_n x_n^k=0\end{array}  \right.\right\}.\]
Assume that $k\geq 2$,  $2\leq m_0\leq \dots\leq m_n$ and that 
\begin{align}\label{eq:maincond}
    \sum_{i=0}^{n} \frac{1}{2 s_0(k m_i)} > 1,
\end{align}
where 
\begin{align*}
    s_0(m) = \min\{ 2^{m-1},  \textstyle{\frac12} m (m-1) + \lfloor \sqrt{2m+2}\rfloor \} \qquad \forall m\in\N.
\end{align*}
    Then there exist constants $C\geq 0$ and $\eta>0$ such that for all $B> 0$ we have
    \begin{equation}\label{eq:asymptotic}
    \# N(X,D,B)=C B^{k\Gamma} + O_\eta\left(B^{k\Gamma -\eta}\right),
    \end{equation}
    where $k\Gamma=\sum_{i=0}^n\frac 1{m_i}-k$ is the Fujita invariant of $(X,D)$ with respect to the line bundle $\mathcal O_X(1)$.
\end{theorem}
An explicit expression for the constant $C$ is given in \eqref{eq: leading constant thm 1.1}. 
Further analysis of the leading constant, including determining under which conditions it is non-zero, will be carried out in a forthcoming paper.

\begin{remark}
For $k=1$, a version of Theorem \ref{thm: diagonal hypersurface} was established in \cite{BY21}, and it is compatible with the PSTVA Conjecture as far as the order of growth is concerned. For $k\geq 2$, the integral model obtained by the closure of $X$ in $\mathbb{P}^n_{\mathbb Z}$ that we use in Theorem \ref{thm: diagonal hypersurface} is not regular. However, the order of growth in Theorem \ref{thm: diagonal hypersurface} is still compatible with the prediction in \cite[Conjecture 1.1]{PSTVA21}. In Section \ref{subsec:regular model} we describe the set of Campana points on the regular model obtained by inverting all the prime numbers dividing $k\cdot c_0\cdots c_n$. The set of Campana points with respect to this model is larger than the set in Theorem~\ref{thm: diagonal hypersurface}, but the corresponding asymptotic, which we compute in a forthcoming paper, is expected to have the same order of magnitude as \eqref{eq:asymptotic}.
\end{remark}

\begin{remark}Browning and Yamagishi \cite{BY21} had the condition $\sum_{i=0}^{n-1} \frac{1}{m_i(m_i+1)} \ge 1$ in their result. This seems  stricter than necessary, since on the one hand it isolates the last coefficient, creating a somewhat artificial condition, and on the other hand it fails to take advantage of the strongest available bounds from \cite[Section 14]{Woo19}. In our result, we address both of those issues in the case $k \ge 2$. In principle, the methods can be extended also to the case $k=1$, but this would require a more careful treatment in some parts of the argument. Since the main focus of our paper is on equations of degree $2$ or higher, we did not expend the extra effort that would be required to achieve a strengthening of the result of \cite{BY21}. For context, in the smallest possible case that is admissible within Theorem~\ref{thm: diagonal hypersurface}, namely $k=m_0=\ldots=m_n=2$, we have $s_0(km_i)=s_0(4)=8$, and consequently the condition \eqref{eq:maincond} translates into the bound $n \ge 16$.

Although it improves the range of applicability compared to \cite{BY21}, the condition in our theorem is still far from the log Fano condition $\sum_{i=0}^n\frac 1{km_i}>1$, which is the geometric condition under which it makes sense to expect an asymptotic as predicted by the PSTVA Conjecture. In fact, as far as applications of the circle method involving mean value estimates are concerned, the limit of the method is at $\sum_{i=0}^n\frac 1{km_i}>2$.   
\end{remark}

\subsection{Outline of the paper}
The paper is organized as follows. In Section \ref{sec: circle method} we give a brief history of the circle method and explain the main ideas of the technique. We also give an overview of applications of the circle method to the problem of counting Campana points of bounded height. In Section~\ref{sec: Campana points} we define Campana points, and give an overview of different definitions that have been used in the literature. We focus on diagonal hypersurfaces in Section \ref{sec: campana points on diagonal hypersurfaces}, where we describe sets of Campana points of bounded height and set up the counting problem for Theorem \ref{thm: diagonal hypersurface}. In Section \ref{sec: circle method result} 
we develop a variant of the circle method of \cite{BY21} combined with the bounds in \cite[Section 14]{Woo19}, and we finally prove Theorem \ref{thm: diagonal hypersurface} in Section \ref{sec:proof main theorem}.

\subsection{Notation} 
\label{sec: notation}
Given $m\in\N$ and a set $S$ of prime numbers, an integer $x$ is \textsl{$m$-full outside $S$} if  $p^m|x$ holds for every prime $p\notin S$ dividing $x$; we say that $x$ is \textsl{$m$-full} if we can take $S=\emptyset$. To avoid confusion, we stress here that we do not consider $0$ to be an element of $\N$.

For a number field $F$, and for a finite set $S$ of places of $F$ containing all infinite places, we denote by $\mathcal{O}_{F,S}$ the ring of $S$-integers, that is, all elements of $F$ with non-negative $v$-adic valuation for all places $v\notin S$. We denote by $\infty$ the infinite place of $\Q$. For a prime number $p$, we denote by $\val_p$ the $p$-adic valuation.

In Sections \ref{sec: campana points on diagonal hypersurfaces}, \ref{sec: circle method result} and \ref{sec:proof main theorem}, we denote tuples of integers by bold letters as follows: $\bc=(c_0,\dots,c_n) \in \Z^{n+1}$, and similarly $\bd,\beps, \bs, \bx, \tilde\bu \in \Z^{n+1}$. We further have $\bt=(t_{i,r})_{0\leq i\leq n, 1\leq r\leq m_i-1} \in\N^{\sum_i m_i - n}$ for $m_i$ as in Theorem~\ref{thm: diagonal hypersurface}, and similarly $\tilde \bv \in \N^{\sum_i m_i - n}$. 
Throughout, we assume that $B$ is a large positive real number.
For $\bx\in \Z^{n+1}$, we write $|\bx|=\max_{0\leq i\leq n}|x_i|$. 
We denote by $\Z_{\textup{prim}}^{n+1}$ the set $\{\bx\in\Z^{n+1}:\gcd(x_0,\dots,x_n)=1\}$, by $\prod_p$ a product over all prime numbers, and by $\mu$ the M\"obius function. 
For $\alpha\in\mathbb R$, we write $e(\alpha) = \exp(2\pi i\alpha)$, and we denote by $\| \alpha \|$ the distance between $\alpha$ and its closest integer.

Throughout, the implicit constants in the estimates $\ll$ and $O(\cdot)$ are allowed to depend on the fixed data $k$, $\epsilon$ and $\delta$, as well as the coefficients $c_0,\dots,c_n$ and $d_0,\dots,d_n$, and the parameters  $m_0,\dots,m_n$ and $\tilde m_0,\dots,\tilde m_n$. In particular, all statements involving the symbol $\epsilon$ are asserted to hold for all $\epsilon>0$. Here, we do not track the precise `value' of $\epsilon$, which consequently may vary from one expression to the next.

\subsection{Acknowledgements}
We want to thank the organizers of the workshop Women in Numbers Europe 4 which took place in the summer of 2022 in Utrecht, The Netherlands, which is where this project originated. We thank Tim Browning for useful comments. While working on this paper Judith Ortmann was partially supported by the Caroline Herschel Programme of Leibniz University Hannover, as well as a scholarship for a research stay abroad of the Graduiertenakademie of Leibniz University Hannover since part of the work was done while visiting the University of Bath. Rosa Winter was supported by UKRI Fellowship MR/T041609/1. In the final stages of the work, Julia Brandes was supported by Project Grant 2022-03717 from Vetenskapsr{\aa}det (Swedish Science Foundation), and Marta Pieropan was supported by the grant VI.Vidi.213.019 of the Nederlandse Organisatie voor Wetenschappelijk Onderzoek. Part of the work was completed while Julia Brandes was visiting the Max Planck Institute for Mathematics in Bonn, whose generous support is also gratefully acknowledged. 
Last but not least, we are grateful to Sam Streeter as well as two anonymous referees for their useful comments.

\section{The circle method}\label{sec: circle method}

\subsection{History of the circle method}
In the 1920's Hardy and Littlewood developed the so called `Hardy-Littlewood circle method' (sometimes called just `Hardy-Littlewood method' or `circle method'), which is an analytic method to treat additive problems in number theory.
These problems deal with the representation of a large number as a sum of numbers of some specified type. The most famous additive problem is Waring's problem:
Let $k$ be a positive integer. Can every large integer $N$ be written as a sum of a bounded number of $k$-th powers
\begin{equation}\label{eq: Waring equation}
    N = x_1^k + \dots + x_s^k,
\end{equation}
where $s$ is a positive integer?
Let $r_{k,s}(N)$ denote the number of such representations of $N$. Then, Waring's problem is equivalent to showing that $r_{k,s}(N)>0$ for some $s$ and all sufficiently large integers $N$. Waring's problem was first proved by Hilbert in 1909. From 1918 to 1920, 
Hardy and Littlewood together with Ramanujan were the first to prove an explicit upper bound for $r_{k,s}(N)$ by using a new analytic method. Their work laid down the foundations of the circle method in its original form (see \cite{Nathanson96, davenport_browning_2005}).

It turned out that the circle method is a very powerful method, since it can be used to solve a diverse range of number theoretic problems. For example, it is one of the most significant all purpose tools when studying rational points on higher-dimensional algebraic varieties. We refer to \cite{browning09} for an overview.

In its earliest versions, the starting point of the circle method is the generating function
\[f(z) = \sum_{a\in A} z^a \quad (\vert z\vert<1),\]
where $A$ denotes a set of nonnegative integers. 
Taking $f(z)$ to the $s$th 
power yields the series
\[ f(z)^s = \sum_{N=0}^{\infty} R_{A,s}(N)z^N,\]
whose coefficients $R_{A,s}(N)$ 
encode the number of solutions of the equation 
\begin{equation}\label{eq: Waring-eq 2}
    N = a_1 + \dots + a_s
\end{equation}
with $a_1,\dots,a_s\in A$. We can isolate the $N$th coefficient by means of Cauchy's integral formula, which yields
\begin{equation}\label{eq: Waring rep old}
    R_{A,s}(N) = \frac{1}{2\pi i}\int_{\vert z\vert = \rho} \frac{f(z)^s}{z^{N+1}}\mathrm{d}z
\end{equation}
for any $\rho\in(0,1)$. In the original (Hardy--Ramanujan) version of the circle method this integral is evaluated by dividing the circle of integration into two disjoint sets, the `major arcs' and the `minor arcs'. Classically, the major arcs contribute to the main term and the minor arcs to the error term.

In 1928 Vinogradov introduced a helpful simplification by transferring the problem from complex analysis to Fourier analysis. 

The Fourier transform maps the set of integers $\Z$ onto the real unit interval $\R/\Z \simeq [0,1)$. For any integer $N$ we denote as before by $R_{A,s}(N)$ the number of representations of $N$ as a sum of $s$ elements in a finite set $A$. The inverse Fourier transform $\mathcal F^{-1} R_{A,s}$ of $R_{A,s}$ is then a function of $\alpha$ and is given by 
\[
    \mathcal F^{-1} R_{A,s} (\alpha) = \sum_{N \in \Z}R_{A,s}(N) e(\alpha N) = \sum_{a_1, \ldots, a_s \in A} e(\alpha(a_1+\dots+a_s)) = \left(\sum_{a \in A}e(\alpha a) \right)^s.
\]

It is convenient to put 
\[
    F(\alpha) = \sum_{a \in A} e(\alpha a),
\]
and the reader may note that with this definition we have $F(\alpha)=f(e(\alpha))$ where $f$ is the generating function considered above. 
We can now apply the forward Fourier transform and obtain 
\[
    \mathcal F \mathcal F^{-1} R_{A,s}(N) = \int_0^1 F(\alpha)^s e(-\alpha N) \mathrm{d} \alpha. 
\]
By the Fourier Inverse Theorem, we have $\mathcal F \mathcal F^{-1} R_{A,s}(N)=R_{A,s}(N)$, and thus we obtain the formula 
\begin{equation}\label{eq: Waring rep} 
    R_{A,s}(N) = \int_0^1 F(\alpha)^s e(-\alpha N) \mathrm{d}\alpha. 
\end{equation}

Clearly, since we assume $A$ to consist of non-negative integers, the equation \eqref{eq: Waring-eq 2} implies that without loss of generality the sum $F(\alpha)$ can be truncated at $a \le N$, as larger values trivially cannot contribute. 
We remark that if the set $A$ is allowed to include negative numbers as well, one typically includes some other truncation $a \le B$ for some large parameter $B$.

\subsection{The modern circle method: Main steps}

We are interested in the situation when the set $A$ is the set of $k$-th powers. In that case the exponential sum is given by 
\[F(\alpha) = \sum_{1 \le x \le P} e(\alpha x^k),\] 
where we put $P = \lfloor N^{1/k}\rfloor$. 

The main strategy is now to try to understand the integral in \eqref{eq: Waring rep} by studying the size of the integrand, and in particular the size of the exponential sums $F(\alpha)$ as $\alpha$ ranges over the unit interval.  For a typical (irrational) $\alpha$, the individual summands within the exponential sum $F(\alpha)$ are more or less equidistributed over the unit circle, and consequently the behaviour of the exponential sum $F(\alpha)$ is reminiscent of Brownian motion. In fact, \cite[Corollary 2.2]{CS} shows that $F(\alpha) \ll P^{1/2 + \epsilon}$ for all $\alpha$ in a set $\mathcal{L} \subset [0,1]$ with Lebesgue measure $1$. At the same time, when $\alpha$ is a rational number with denominator $q$, the summands inside $F(\alpha)$ can take at most $q$ distinct values on the unit circle. Thus, in particular for small values of $q$ there is significant potential for interference. The extreme case of this is when $\alpha=0=\frac01$, which gives the value $F(0)=P$. Moreover, since $F$ is continuous, this interference behaviour extends to suitably small neighbourhoods of rational numbers with small denominator.

This motivates the dissection of the interval $[0,1]$ into major and minor arcs, where the major arcs $\mathfrak{M}$ comprise all $\alpha$ that are close to a rational number with small denominator, so that $F(\alpha)$ is potentially large, and the minor arcs $\mathfrak{m}$ collect the remaining $\alpha$ that lack a similarly strong rational approximation.

The treatment of the major arcs is by now mainly standard. By using the fact that $\alpha\in\mathfrak{M}$ has a good approximation by a rational number $a/q$, we can write $\alpha = a/q + \theta$ for some small $\theta$. This leads to a factorization of the integral over the major arcs into a product of two terms: the `singular integral' and the `singular series', which then can be evaluated  to obtain the main term. Here, the singular integral can be interpreted as the volume of the solution set of the equation \eqref{eq: Waring equation} when viewed as a submanifold inside $\mathbb R^s$; this also provides the expected order of growth. Meanwhile, the singular series encodes all congruence information related to the equation \eqref{eq: Waring equation}. In fact, it can be factorised into an Euler product, where each factor describes the volume of the solution set of \eqref{eq: Waring equation} when viewed as a subset of the $p$-adic numbers $\mathbb Q_p$ as $p$ ranges over the primes. Thus, the canonical outcome of the circle method is an asymptotic formula with a main term of size $\asymp P^{s-k} \sim N^{s/k-1}$, which is modulated by a product of local factors that encode any local obstructions the problem might have. 

The bottleneck of the problem is the treatment of the minor arcs. The main difficulty is that although we have good control over the size of $F(\alpha)$ in an almost-all sense, our understanding in a \emph{pointwise} sense remains comparatively poor. Fortunately, however, for additive problems like the one in \eqref{eq: Waring equation} the contribution to \eqref{eq: Waring rep} that stems from the minor arcs can be estimated by 
\[
    \int_{\mathfrak m} F(\alpha)^se(-n\alpha) \mathrm{d}\alpha \le \sup_{\alpha \in \mathfrak m} |F(\alpha)| \int_0^1 |F(\alpha)|^{s-1} \mathrm{d}\alpha.
\]
In other words, it becomes less crucial to have strong pointwise bounds for $F(\alpha)$ if in addition we have good control over the \emph{average} behaviour of moments of $F$. Since $F$ exhibits square-root cancellation almost everywhere, this is a more tractable problem at least for small moments. For larger moments, however, this problem still presents formidable difficulties. 

A major breakthrough in the field was the resolution of the main conjecture associated with Vinogradov's mean value theorem (\cite{BDG,Woo16,Woo19}), which gives a near complete understanding of the average behaviour of moments of the related exponential sum 
\[
    G(\alpha) = \sum_{1 \le x \le P} e(\alpha_1 x + \alpha_2 x^2 + \cdots + \alpha_k x^k).
\]
In fact, the strongest available bounds for mean values of moments of $F$ are derived from these results (see \cite[Section~14]{Woo19}). These are the bounds we will exploit in our arguments below.

\subsection{Examples of counting Campana points by the circle method}\label{subsec: applications circle method to campana points} 
Here, we summarize instances in which the circle method is used to count Campana points of bounded height. So far, there are only a few examples of this. 

Van Valckenborgh and Browning were the first to use the circle method to count Campana points of bounded height. Much of the early investigations is centered around the set of Campana points for the boundary divisor $\frac 12\sum_{i=0}^n\{x_i=0\}$, corresponding to squareful numbers.  In this setting,  Van Valckenborgh \cite{VV12} gave an asymptotic formula for the number of integral points $(a_0:\dots:a_n)$ of bounded height on the hyperplane $\{\sum_{i=0}^n x_i=0\}$ in $\PP^n$ such that $a_i$ is squareful, provided that $n \ge 4$. 

The picture becomes more complicated for $ n \le 3$, as underlying geometric considerations begin to have an effect. In \cite{BVV12} Browning and Van Valckenborgh considered the case $n=2$ and investigated the number of positive coprime squareful numbers $x,y,z$  of bounded height $B$ that satisfy the equation $x+y=z$. In particular, they give a lower bound of order $O(B^{1/2})$ for the number of corresponding Campana points, which they conjecture to be sharp. Unfortunately, establishing the corresponding upper bound turns out to be significantly harder, and here they use a method rooted in the determinant method to give an upper bound of order $O(B^{3/5 + \epsilon})$. 
Finally, Shute \cite{Shute21} settled the case $n=3$ by adapting the delta-symbol method \cite{HB:96}. 
In this context, it is noteworthy that in order to put the problem in the framework of Manin's conjecture, the author has to exclude certain accumulating special subvarieties from the count, reinforcing the principle that geometric phenomena tend to have a disproportionate effect in small dimension. We can therefore summarise that in the setting of counting Campana points with boundary divisor $\frac 12\sum_{i=0}^n\{x_i=0\}$ on the hyperplane $x_0+ \cdots + x_n=0$ only the case $n=2$ is still open.

Meanwhile, Browning and Yamagishi generalized the work of Van Valckenborgh \cite{VV12} in \cite{BY21} by extending it to more general boundary divisors. Instead of squareful integers $a_i$ they consider more generally $m_i$-full integers, where $m_i\ge 2$. They further generalize 
to hypersurfaces $c_0x_0+\dots+c_{n-1}x_{n-1} = x_n$ 
for fixed nonzero integers $c_0,\dots,c_{n-1}$.
Similarly to \cite{VV12}, they give an asymptotic formula for the number of Campana points of bounded height in this setting by using the circle method under the assumption that the following condition is satisfied:
\[\sum_{\substack{0\le i\le n\\i\neq j}}\frac{1}{m_i(m_i+1)}\ge 1 \text{ for some } 0\le j\le n.\]
It turns out that considering these Campana points can be interpreted as studying Waring's problem for mixed exponents. Browning and Yamagishi used this to further prove an asymptotic formula for Waring's problem of mixed powers.

In this paper we extend the result of \cite{BY21} to diagonal hypersurfaces of degree $\geq 2$.

\section{Campana points}\label{sec: Campana points}

In this section we introduce the notion of Campana points that we 
use in Theorem \ref{thm: diagonal hypersurface}.
Campana introduced the notion of `orbifoldes g\'{e}om\'{e}triques' and `orbifold rational point' in his papers \cite{Ca04,Ca05,Ca11,Ca15}. Since then, several different definitions of Campana points appeared in the literature \cite{Ab09,AVA18}, which agree with the original definition of Campana on curves. As is explained in \cite{PSTVA21}, this is not the case for higher dimensional varieties, where the different definitions can lead to significant differences in the associated counting problems. All the papers just mentioned define Campana points on a regular integral model of the variety over the ring of $S$-integers of the field of definition, for some finite set of places $S$. Very recently, Mitankin, Nakahara and Streeter gave a definition of Campana points for \textsl{nonregular} models \cite{MNS22}. We will use this definition in this paper, since it allows us to directly interpret the counting problem in Theorem \ref{thm: diagonal hypersurface} 
in the setting of Campana points.

\begin{defn}
Let $F$ be any field. 
A \textsl{Campana orbifold over $F$} is a pair $(X,D)$ where
\begin{itemize}
    \item $X$ is a smooth proper variety over $F$, and 
    \item $D$ is an effective Weil $\Q$-divisor on $X$ defined over $F$ 
    satisfying
    $$D = \sum_{\alpha \in \calA} \epsilon_\alpha D_\alpha,$$
    where $\mathcal{A}$ is a finite index set, the $D_\alpha$'s are distinct prime divisors on $X$, and $\epsilon_\alpha$ belongs to the set
    \begin{equation*}
    \calM = \left\{ \left. 1 - \frac{1}{m} \right|  m \in \N, m\geq 2 \right\}.\end{equation*}
    
    Note that the $D_\alpha$'s are irreducible over $F$, but not necessarily over $\overline{F}$.
  We define $D_{\mathrm{red}}= \sum_{\alpha \in \calA} D_\alpha$, and say that $(X,D)$ is \textsl{smooth} if $D_{\mathrm{red}}$ is a strict normal crossing divisor on $X$. 
\end{itemize}
\end{defn}

\begin{remark}
In the existing literature, the set $\calM$ 
is usually taken to include 0 and 1. For the purposes of this paper we don't need these values, hence we omit them in our definition.
\end{remark}

Since Campana points depend on the choice of an integral model for the Campana orbifold, we need to define what a good choice for such an integral model is. From now on, we take $F$ to be a number field.

\begin{defn} Let $(X, D)$ be a Campana orbifold over $F$. 
Let $S \subset \Omega_F$ be a finite set of places of $F$ containing all the infinite places.
A \textsl{proper integral model} of $(X, D)$ over $\mathcal O_{F,S}$ is a pair $(\mathcal X,\mathcal D)$ such that $\calX$ is a flat proper model of $X$ over $\mathcal O_{F,S}$, and $\calD = \sum_{\alpha \in \calA} \epsilon_\alpha \calD_\alpha$, where $\calD_\alpha$ is the Zariski-closure of $D_\alpha$ in $\calX$. If $\mathcal X$ is regular, we say that $(\calX, \calD)$ is a \textsl{good integral model} of $(X,D)$.
\end{defn}

Finally, before we can define Campana points, we need to understand the `intersection behaviour' 
of a given point $P \in X(F)$ with the divisor $D$.

\begin{defn}\label{intersec}
Let $(X,D)$ be a Campana orbifold over $F$, and let $(\calX, \calD)$ be a proper 
integral model over $\calO_{F,S}$.
Take $P \in X(F)$. Since $\calX$ is proper, we have $X(F) = \calX(\calO_{F,S})$ and so $P$ extends uniquely to a point $\calP \in \calX(\calO_{F,S})$.  
Fix $v \notin S$ and let $\calP_v \in \calX(\calO_v)$ be the point corresponding to $P$. By definition, we can also view it as a map $\calP_v : \Spec(\calO_v) \to \calX$. Fix $\alpha \in \calA$. Since $\calD_{\alpha} \subset \calX$ is a closed subscheme, the fiber product $\calD_{\alpha} \times_{\calX} \Spec(\calO_v) \subset \Spec(\calO_v)$ is a closed subscheme corresponding to a non-zero ideal $I_{v,\mathcal P,\alpha}$ in $\calO_v$, using the correspondence between closed subschemes of $\Spec(\calO_{v})$ and ideals in $\calO_v$.  We distinguish two cases: 
\begin{itemize}
\item If $\calP_v\not\subseteq\mathcal{D}_{\alpha}$,   we define the
    \textsl{intersection multiplicity of $P$ and $\calD_\alpha$ at $v$} to be
    \[ n_v(\calD_\alpha, P) = \textrm{length}(\calO_v/I_{v,\mathcal P,\alpha}).\]
    Equivalently, since $\calO_v$ is a discrete valuation ring, say with a choice of uniformiser $\pi_v$ for its maximal ideal, it follows that the non-zero ideal $I_{v,\mathcal P,\alpha}$ can be written as
    \[ I_{v,\mathcal P,\alpha} = \left( \pi_v^{}\right)^{n_v(\calD_\alpha, P)}. \]
\item If $\calP_v\subseteq\mathcal{D}_{\alpha}$,  we set the \textsl{intersection multiplicity of $P$ and $\calD_\alpha$ at $v$} to be $+\infty$ (in this case, the ideal $I_{v,\mathcal P,\alpha}$ is just the zero ideal $(0)$ in $\calO_v$).
\end{itemize}

To summarise, the \textsl{intersection multiplicity of $P$ and $\calD_\alpha$ at $v$} is 
\[ 
n_v(\calD_\alpha, P) = 
\begin{cases}
 \textrm{length}(\calO_v/I_{v,\mathcal P,\alpha}) & \textrm{ if $\mathcal{P}_v \not\subseteq \mathcal{D}_{\alpha} $},\\
 + \infty & \textrm{ if $\mathcal{P}_v \subseteq \mathcal{D}_{\alpha} $}.
\end{cases}\]
\end{defn}

\begin{remark}\label{rem: intersection in regular models}
In practice, given a point $\calP_v:\Spec(\mathcal O_v)\to\calX$, and an open subset $U\subseteq\calX$ containing the image of $\calP_v$, such that $\mathcal D_\alpha|_{U}$ is principal and defined by a rational function $f_{\alpha,U}$ which is regular on $U$ (i.e., $\mathcal{D}_{\alpha}$ is Cartier in a neighbourhood of $\mathcal{P}_{v}$), we have $n_v(\calD_\alpha,P)=\val_v(f_{\alpha,U}(\calP_v))$, where $f_{\alpha,U}(\calP_v)$ is the image of $f_{\alpha,U}$ in $\calO_v$ under the ring homomorphism that defines $\calP_v:\Spec(\mathcal O_v)\to U$, and $\val_v$ is the $v$-adic valuation on $\mathcal{O}_v$. 
\end{remark}

We are now in the position to define Campana points.

\begin{defn}\label{def: campana points} We keep the notation as in Definition \ref{intersec}. For $\alpha\in\mathcal{A}$ we write $\epsilon_{\alpha}=1-\frac{1}{m_\alpha}$. A point $P \in X(F)$ is a  \textsl{Campana $\calO_{F,S}$-point on $(\calX, \calD)$} if, for every place $v \notin S$ and every $\alpha \in \calA$, we have  
either $n_v(\calD_\alpha, P) = 0$ or $n_v(\calD_\alpha,P) \geq m_\alpha$.
\end{defn}

\begin{remark} Different choices of the finite set of places $S \subset \Omega_F$ lead to potentially different sets of Campana points, and thus to potentially different counting problems.    
\end{remark}

\begin{example}\label{example}
Using Remark \ref{rem: intersection in regular models}, we describe Campana points for $F = \Q$ and $X\subseteq\mathbb{P}^n_{\Q}$, in the setting where the divisor $D_\alpha$ is the restriction to $X$ of a divisor $D'_\alpha\subseteq\mathbb P^n$ for all $\alpha\in\mathcal{A}$, and we have a model $\calX\subseteq\mathbb P^n_{\Z_{S}}$ for some set of places $S$ (for example, when $\calX$ is the Zariski closure of $X$ in $\mathbb P^n_{\Z_{S}}$). For all $\alpha\in\mathcal{A}$, let $f_\alpha \in\Z[x_0,\dots,x_n]$ be homogeneous with coprime coefficients (i.e., with content~1) such that $D_\alpha$ is defined by $f_\alpha$. Fix a point $P\in X(\Q)$ and write $P=(a_0:\dots:a_n)$ 
in projective homogeneous coordinates with $a_0,\dots,a_n\in \Z$ such that $\gcd(a_0,\dots,a_n)=1$. Let $\ell_P\in\Z[x_0,\dots,x_n]$ be a linear form with $\ell_P(a_0,\dots,a_n)=1$. Then the image of $\mathcal P$ in $\mathcal{X}$ is contained in the affine patch $U=\mathbb P^n_{\Z_s}\smallsetminus \{\ell_P=0\}$, so for all $\alpha\in\calA$ we can take $f_{\alpha,U}=f_\alpha/\ell_P^{\deg f_\alpha}$, and find $n_p(\calD_\alpha,P)=\val_p(f_{\alpha}(a_0,\dots,a_n))$ for every prime $p\notin S$. It follows that $P \in X(\Q)= \calX(\Z_S)$ is a Campana point precisely when, for all $\alpha \in \calA$, we have that $f_\alpha(a_0,\dots,a_n)$ is $m_\alpha$-full outside $S$.
\end{example}

\section{Campana points on diagonal hypersurfaces}\label{sec: campana points on diagonal hypersurfaces}
In this section we set up the problem of counting Campana points of bounded height on diagonal hypersurfaces in $\mathbb{P}^n_{\Q}$ with boundary divisor induced by the coordinate hyperplanes. We do this in two settings: first by constructing a proper integral model over $\mathbb{Z}$, and then by constructing a good integral model over $\mathbb{Z}_S$ for a suitable finite set of places $S$.
\subsection{A proper model}\label{sec: proper model}

Let $F = \Q$ and let $k \in \Z_{\geq 1}$. Let $X \subset \mathbb{P}^n_\Q$ be the diagonal hypersurface given by the equation
\[ \sum_{i=0}^n c_i x^k_i = 0,\]
where $c_0,\dots,c_n \in \Z_{\neq0}$, 
and $\gcd(c_0,\dots,c_n)=1$. We consider the effective $\Q$-divisor on $X$ given by 
\[ D = \sum_{i=0}^n\epsilon_iD_i,\]
where $D_i = \{ x_i = 0\}\cap X$ and $\epsilon_i$ belongs to the set $\calM$ 
for all $i\in\{0,\dots,n\}$. Then $(X,D)$ is a Campana orbifold over $\mathbb{Q}$. Write $\epsilon_i=1-\frac{1}{m_i}$ for all $i\in\{0,\dots,n\}$.
\begin{defn}
We denote by $(\mathcal{X}_1,\mathcal{D}_1)$ the proper integral model of $(X,D)$ given by the same equations for $X$ and $D$ over $\mathbb{Z}_S$, where $S=\{\infty\}$.  
\end{defn}
\begin{lemma}\label{lem: set of campana points}Let $B>0$ be a real number. With the model $(\mathcal{X}_1,\mathcal{D}_1)$ and the height induced by the Weil height on $\mathbb{P}^n(\mathbb{Q})$, the set of Campana $\mathbb{Z}$-points  on $X\cap \left(\mathbb{P}^n\smallsetminus \bigcup_{i=0}^n \{x_i=0\} \right)$ 
 of height bounded by $B$ is given by 
\begin{equation}\label{eq:set_camp_points}
   N(X,D,B) =  \left\{ (x_0: \dots: x_n) \in \mathbb P^n(\mathbb Q) \left| \begin{array}{c} 
   x_0,\dots,x_n\in \mathbb Z_{\neq 0},\\
   \gcd(x_0, \dots, x_n)=1,\\    
    x_i \textup{ is } m_i \textup{-full}\ \forall i\in\{0,\dots,n\},\\ 
    |x_0|, \dots,|x_n| \le B, \\
    c_0x_0^k+\dots +c_n x_n^k=0\end{array}  \right.\right\}.
\end{equation}
\end{lemma}
\begin{proof}
Take $a_0,\ldots,a_n\in\mathbb{Z}_{\neq0}$ with $\gcd(a_0, \ldots, a_n) =1$, such that $P=(a_0:\cdots: a_n)$ is contained in $X\cap \left(\mathbb{P}^n\smallsetminus \bigcup_{i=0}^n \{x_i=0\} \right)$ and of height at most $B$. The only non-obvious condition to check in order to prove the lemma is that $P$ satisfies Definition~\ref{def: campana points} if and only if the third condition in $N(X,D,B)$ holds. Fix $i$ in $\{0,\ldots,n\}$. As in Example~\ref{example}, let $\ell_P$ be a linear form in $\mathbb{Z}[x_0,\ldots,x_n]$ with $\ell_P(a_0,\ldots,a_n)=1$. Then the divisor $\calD_i$ is Cartier on the open set $U=\mathbb{P}^n\smallsetminus\{\ell_P=0\}$, given by the rational function $x_i/\ell_P$. We find $n_p(\calD_i,P)=\text{val}_p(a_i)$, and by definition, $P$ is a Campana point if and only if $\text{val}_p(a_i)\in\mathbb{Z}_{\geq m_i}\cup\{0\}$, which is equivalent to $a_i$ being $m_i$-full. \end{proof}

\subsection{Bad primes and a regular model}\label{subsec:regular model}

The model $(\calX_1,\calD_1)$ over $\mathbb{Z}$ is not a good integral model of $(X,D)$ in general, since $\calX_1$ is not regular as a scheme over $\mathbb{Z}$. The following lemma gives sufficient conditions on a finite set $S$ of primes for $(\calX_1,\calD_1)$ to be \textsl{smooth} over $\mathbb{Z}_S$.

\begin{lemma}\label{lem: smooth model}
For $S$ a finite set of places of $\mathbb{Q}$ containing $\infty$ and all primes dividing $k$ and all $c_i$, the model $(\calX_2,\calD_2)$ of $(X,D)$ over $\mathbb{Z}_S$ defined by the same equations as $X$ and $D$ over $\mathbb{Z}_S$ is smooth. 
\end{lemma}

\begin{proof} 
Let $\calX \subset \PP^n_\Z$ be the projective scheme over $\Spec \, \Z$ defined by the equation
\[ c_0 x_0^k + \cdots + c_n x_n^k = 0.\]
Notice that the structure morphism $s : \calX \to \Spec\, \Z$ is flat by \cite[Proposition 4.3.9]{Liu}, as $\calX$ is integral because $\gcd(c_0,\dots,c_n)=1$, $\Spec \,  \Z$ is a Dedekind domain, and $s$ is non-constant. Let $S$ be a finite set of places of $\Q$ containing $\infty$ and all the primes $p$ which divide $k \cdot \prod_{i=0}^n c_i$. Consider $\calX_S = \calX \times_{\Spec \,  \Z} \Spec(\Z_S)$.
Since fiber products preserve flatness, and since $s$ is flat,
the structure morphism $s_S: \calX_S \to \Spec(\Z_S)$ is flat as well.
Since $s_S$ is flat, by \cite[Theorem 10.2]{hartshorne} it suffices to show that the geometric fibers of $s_S$ are regular. The geometric generic fiber $\calX_{\overline \Q}$ is regular by the Jacobian criterion \cite[Exercise I.5.8]{hartshorne}, as $c_0,\dots,c_n\neq 0$, and since any Noetherian scheme is regular if and only if it is regular at its closed points \cite[Chapter 4, Corollary 2.17]{Liu}. For every prime number $p\notin S$, the geometric fiber $\calX_{\overline{\mathbb F}_p}$ over $(p)\in\Spec \, \mathbb Z$ is regular by the Jacobian criterion  \cite[Exercise I.5.8]{hartshorne}, as $p\nmid k\cdot \prod_{i=0}^n c_i$, and since any Noetherian scheme is regular if and only if it is regular at its closed points \cite[Chapter 4, Corollary 2.17]{Liu}.
\end{proof}

From Lemma \ref{lem: smooth model} it follows that we can always obtain a good integral model for $(X, D)$ by setting 
\[ S = \left\{ p \textrm{ prime} : p \  | \ k \cdot \prod_{i=0}^n c_i\right\} \cup \{ \infty\},\]
and defining $(\calX_2, \calD_2)$ over $\mathbb{Z}_S$ by the same equations for $X$ and $D_i$ over $\mathbb{Z}_S$. The set of Campana points on $X$ with respect to the model $\mathcal{X}_2$ is then similar to the set in Lemma \ref{lem: set of campana points}, but requiring the $x_i$-coordinate of the point to be $m_i$-full outside $S$ instead of just $m_i$-full. This leads to a larger set of Campana points than the one in Lemma \ref{lem: set of campana points}.

\vspace{11pt}

It is not obvious that in order to obtain a regular model we need to take the set $S$ as defined above (or larger); for certain choices of $X$ we might get regular models using smaller $S$ as well. However, we remark that primes dividing one or more of the coefficients $c_i$ can be potentially problematic and induce non-regular points on $\mathcal{X}_1$, as the following example shows.

\begin{example}
Let $k \geq 2$ be an integer. Let $p$ be a prime not dividing $k$. Let $\calX$ be defined by 
\[p^kx_0^k -x_1^k + x_2^k =0\subset\mathbb{P}^2_{\mathbb{Z}};\]
it is clear that e.g. the point $(1: p: 0)$ is on $\calX$. Notice that the Jacobian matrix over $\mathbb{F}_p$ vanishes at the point $(1: 0: 0)$, so $\calX$ is not smooth in the fiber over $\mathbb{F}_p$. 
 We now show that $\calX$ is not regular. 
 We work in the affine patch of $\mathbb{P}^2_{\mathbb{Z}}$ given by $\{x_0\neq 0 \}$, and we get the affine equation
 $p^k-y_1^k + y_2^k=0$.
Consider the maximal ideal
\[ \frakm = (p, y_1, y_2) \in \Spec\left( \frac{\Z[y_1, y_2]}{(p^k-y_1^k + y_2^k)}\right),\] which contains, for example, the prime ideal corresponding to the point $(y_1, y_2) = (p,0)$. 
Notice that $p^k-y_1^k + y_2^k$ is contained in $\frakm$. 
We have that $\Spec\left( \frac{\Z[y_1, y_2]}{(p^k-y_1^k + y_2^k)}\right)$ is regular at a point corresponding to $\frakm/(p^k-y_1^k + y_2^k)$ if and only if $p^k-y_1^k + y_2^k \notin \frakm^2$. But, since $k \geq 2$, it is clear that $p^k-y_1^k + y_2^k \in \frakm^2$. Hence, $\Spec\left( \frac{\Z[y_1, y_2]}{(p^k-y_1^k + y_2^k)}\right)$ is not regular at $\frakm$.\end{example}

The primes dividing the exponent $k$ can also be potentially problematic, as the following example shows.
\begin{example}
Let $k \geq 2$ be an integer. Let $p$ be a prime dividing $k$ and write $k = p \lambda$. We consider $\calX$ defined by 
\[x_0^k -x_1^k + x_2^k =0\subset\mathbb{P}^2_{\mathbb{Z}},\]
with the obvious point $(x_0:x_1: x_2) = (1:1: 0)$ on it.  First of all, we notice that the Jacobian matrix over $\mathbb{F}_p$ vanishes identically at every point. We work in the affine patch of $\mathbb{P}^2_{\mathbb{Z}}$ given by $\{x_0\neq 0 \}$, and we get the affine equation $1 -y_1^k + y_2^k=0.$
Consider the maximal ideal
\[ \frakm = (p, y_1 - 1, y_2) \in \Spec\left( \frac{\Z[y_1, y_2]}{(1-y_1^k + y_2^k)}\right),\] which contains, for example, the prime ideal corresponding to the point $(y_1, y_2) = (1,0)$. 

Notice that $1-y_1^k + y_2^k$ is contained in $\frakm$: indeed, $1-y_1^k + y_2^k = (1-y_1) (1+y_1 + \cdots+ y_1^{k-1}) + y_2^k$.
As in the previous example, we have that $\Spec\left( \frac{\Z[y_1, y_2]}{(1-y_1^k + y_2^k)}\right)$ is regular at a point corresponding to $\frakm/(1-y_1^k + y_2^k)$ if and only if $1-y_1^k + y_2^k \notin \frakm^2$. We now show that $1-y_1^k + y_2^k$ is in $\frakm^2$. 
Since $y_2^k \in \frakm^2$ (as $k \geq 2$), it suffices to show that $1-y_1^k \in \frakm^2$.

If $p$ is odd, we notice that
\[(y_1^\lambda -1)^p = y_1^k - 1 + p (-(y_1^\lambda)^{p-1} + y_1^\lambda ) + O(p^2)\]
and thus that
\[ y_1^k - 1 = (y^\lambda_1 - 1)^p +  p y_1^\lambda((y_1^\lambda)^{p-2} - 1 ) + O(p^2). \]
But $(y_1 -1)$ is a factor of $(y_1^\lambda -1)$, and thus $(y_1 -1)^2$ is a factor of $(y_1^\lambda -1)^p$. Moreover, since $\lambda (p-2) > 0$, we have that $(y_1 -1)$ is a factor of $((y_1^\lambda)^{p-2} - 1 )$. Hence, it follows that $1 - y_1^k \in \mathfrak{m}^2$.

If $p = 2$, then completing the square yields
\[ (y^\lambda_1)^2-1 = (y^\lambda_1-1)^2 + 2(y^\lambda_1 -1).\]
But $(y_1 - 1)$ is a factor of $(y_1^\lambda - 1)$. Hence, it follows that $1 - y_1^k \in \mathfrak{m}^2$.

In both cases, we conclude that $1-y_1^k+y_2^k$ is contained in $\mathfrak{m}^2$ and thus that $\Spec\left( \frac{\Z[y_1, y_2]}{(1-y_1^k + y_2^k)}\right)$ is not regular at $\frakm$.
\end{example}

\subsection{The associated counting problem} 
\label{sec: the counting problem}

We now set up the counting problem for Theorem \ref{thm: diagonal hypersurface} using the proper integral model introduced in Section \ref{sec: proper model}. The counting problem using the regular integral model in Section \ref{subsec:regular model} will be carried out in a forthcoming paper.

As in the statement of Theorem \ref{thm: diagonal hypersurface} and in Section \ref{sec: proper model}, we fix nonzero coprime integers $c_0,\dots,c_n$, and integers $k, m_0,\dots,m_n\geq 2$. Up to reordering the indices, we can assume that $m_0\leq \dots \leq m_n$.
We use freely the notation introduced in Section \ref{sec: notation}.
We consider the counting function
\begin{align*}
    N(B) = \# \left \{ \bx\in \Z_{\neq0}^{n+1} \left| \gcd(x_0, \dots, x_n)=1,\, |\bx| \le B,\, x_i \textup{ is } m_i\textup{-full } \forall i\in\{0,\dots,n\}, \,\sum_{i=0}^n c_i x_i^k =0 \right. \right \}
\end{align*}
analogous to \cite[(1.1)]{BY21}. 
Then 
\begin{equation}
\label{eq: step 1/2}
\#N(X,D,B)=\frac 12 N(B),
\end{equation}
as every point in $N(X,D,B)$ has exactly two representatives $(x_0,\dots,x_n)\in\Z^{n+1}$ satisfying the conditions in \eqref{eq:set_camp_points}.

In the rest of this section we rephrase the counting problem in order to apply the circle method. We follow the strategy of \cite[\S3]{BY21}.

Recall that for each $m\in\N$, an $m$-full integer $x \neq 0$ has a unique representation
\begin{align*}
    x=\pm u^m \prod_{r=1}^{m-1} v_r^{m+r},
\end{align*}
with $u, v_1, \dots, v_{m-1} \in \N$ such that $v_1,\dots,v_{m-1}$ are squarefree and pairwise coprime.
Thus, one can rewrite $N(B)$ as the cardinality of the set of tuples $\bx\in \Z_{\neq0}^{n+1}$ satisfying
\begin{gather}
    \gcd(x_0, \dots, x_n)=1, \quad
    |\bx| \le B, \\
    x_i =\pm u_i^{m_i} \prod_{r=1}^{m_i-1} v_{i,r}^{m_i+r} \quad \forall i\in\{0,\dots,n\}, \label{eq:m_j-full}\\
    \mu^2(v_{i,r})=1, \quad \gcd(v_{i,r}, v_{i,\tilde r})=1, \quad \forall i\in\{0,\dots,n\}, \forall r,\tilde r\in\{1,\dots,m_i-1\}, r\neq\tilde r, \label{eq:v_jr_sqfree+gcd}\\
    \sum_{i=0}^n c_i x_i^k =0. \label{eq:equation}
\end{gather}

 Put $\Lambda = \sum_{i=0}^n(m_i-1)$. For integer vectors $\bd = (d_0,\dots,d_n) \in \Z_{\neq 0}^{n+1}$ as well as $\bs=(s_0,\dots,s_n)\in\N^{n+1}$ and $\bt=(t_{i,r})_{0\leq i\leq n, 1\leq r\leq m_i-1}\in\N^\Lambda$, 
let 
\begin{align*}
   N_{\bd}(B, \bs, \bt)
   = \left\{ \bx\in (\N \cap [1, B])^{n+1} \left| \begin{array}{c}
    d_0 x_0^k+ \cdots + d_n x_n^k=0, \ 
    \eqref{eq:m_j-full}, \ \eqref{eq:v_jr_sqfree+gcd}, \\ 
s_i| u_i,\ t_{i,r}| v_{i,r} \ \forall i\in\{0,\dots,n\}, \forall r\in\{1,\dots,m_i-1\} \end{array}\right.\right\},
\end{align*}
and write $N_{\bd}^*(B, \bs, \bt)=N_{\bd}(B, \bs, \bt)\cap \Z_{\textup{prim}}^{n+1}$. 
The sets $N^*_{\bd}(B, \bs, \bt)$ can be used to rewrite $N(B)$ again. It is easy to see that 
\begin{align}
\label{eq: step odd}
    N(B)= \begin{cases} {\displaystyle{\sum_{\beps\in \{\pm 1 \}^{n+1}} \# N_{\beps \bc}^*(B, \bone, \bone ) }}& \text{ if $k$ is odd}, \\
    2^{n+1} \# N^*_{\bc}(B, \bone, \bone )& \text{ if $k$ is even},
    \end{cases}
\end{align}
where we wrote $\beps \bc = (\varepsilon_0 c_0, \ldots, \varepsilon_n c_n)$.
The counting function $N^*_{\bd}(B, \bone,\bone)$ can be expressed in terms of $N_{\bd}(B, \bs, \bt)$ by means of the inclusion-exclusion principle. Indeed, it was shown in \cite[Section~3]{BY21} that there exists a function $\varpi: \N^{n+1} \times \N^\Lambda \to \Z$ such that  
\begin{align}
\label{eq: step prim}
    \# N_{\bd}^* (B, \bone, \bone) = \sum_{\bs\in\N^{n+1}} \sum_{\bt\in\N^\Lambda} \varpi(\bs, \bt) \# N_{\bd} ( B, \bs, \bt ).
\end{align}
Moreover, they show that the function $\varpi$  has the following properties. 
\begin{lemma}[Lemma~3.2 in \cite{BY21}]\label{lem: varpi}
    Let $(\bs,\bt) \neq (\bone,\bone)$.
    Then the following are true.
    \begin{enumerate}
        \item $\varpi(\bs,\bt)=0$ if any one of the coordinates of $\bs$ or $\bt$ is divisible by $p^2$; 
        \item $\varpi(\bs,\bt)=0$ if one of the coordinates of $\bs$ or $\bt$ is divisible by $p$, but there exists $0 \le i \le n$ with $p \nmid s_i t_{i,1}\cdots t_{i,m_i-1}$.
    \end{enumerate}
    Moreover, $\varpi(\bs,\bt) \ll (|\bs| |\bt|)^\epsilon$ for all $\bs \in \N^{n+1}$, $\bt \in \N^{\Lambda}$ and $\epsilon>0$, where the implicit constant can depend on $\epsilon$.
\end{lemma}

   For every $\bx\in N_{\bd}(B,\bs,\bt)$, we write $u_i = s_i\tilde u_i$ and $v_{i,r} =  t_{i,r}\tilde v_{i,r}$ for  $i\in\{0,\dots,n\}$ and $r\in\{1,\dots,m_i-1\}$. Let 
    \begin{align}\label{eq: gammaj}
    \gamma_i = s_i^{km_i} \prod_{r=1}^{m_i-1} t_{i,r}^{k(m_i+r)} \tilde v_{i, r}^{k(m_i+r)} \qquad \forall i\in\{0,\dots,n\}.
\end{align}
Then
\begin{align*}
    x_i^k =\gamma_i \tilde u_i^{k m_i} \qquad \forall i\in\{0,\dots,n\}.
\end{align*}

For a vector $\bzeta \in \N^{n+1}$, let

\begin{align*}
M_{\bd, \bzeta}(B^k)=\# \left \{ \tilde \bu \in \N^{n+1} \left| \zeta_i \tilde u_i^{k m_i} \le B^k \ \forall i\in\{0,\dots,n\}, \sum_{i=0}^n d_i \zeta_i \tilde u_i^{k m_i}=0\right. \right \}.
\end{align*}
For any pair $(\bs,\bt) \in \N^{n+1} \times \N^\Lambda$ and any $R \in \R_{>0} \cup \{\infty\}$ define the set 
\begin{align*}
    V_R(\bs,\bt) = \left\{\tilde \bv\in\N^\Lambda\left| \begin{array}{c} s_i^{km_i} \prod_{r=1}^{m_i-1} t_{i,r}^{k(m_i+r)} \tilde v_{i, r}^{k(m_i+r)} \le R^k \quad \forall i \in \{0,\ldots, n\}\\ \mu^2(\tilde v_{i,r} t_{i,r})=1 \quad \forall i\in \{0,\ldots, n\}, r\in \{1,\ldots, m_i-1\}\\ \gcd(\tilde v_{i,r} t_{i,r},\tilde v_{i,r'} t_{i,r'})=1 \quad \forall i, r, r' \text{ with } r \neq r'\end{array} \right.\right\}.
\end{align*}

Thus 
\begin{align*}
    \# N_{\bd} ( B, \bs, \bt )= \sum_{ \tilde \bv \in V_B(\bs, \bt) } M_{\bd, \bgamma}(B^k),
\end{align*}
where $\bgamma$ is given by \eqref{eq: gammaj}.

For $R \in \R_{>0} \cup \{\infty\}$ we define further
\begin{align*}
    \calT_R = \left\{\bs\in\N^{n+1}, \bt\in\N^\Lambda \left| \begin{array}{c}
    s_i^{m_i} \prod_{r=1}^{m_i-1} t_{i,r}^{m_i+r} \le R \\
    \mu^2(s_i)=\mu^2(t_{i,r})=1 \quad \forall i\in \{0,\ldots, n\}, r\in \{1,\ldots, m_i-1\}\\
    \gcd(t_{i,r}, t_{i,r'})=1 \quad \forall i\in \{0,\ldots, n\}, r \neq r'\in \{1,\ldots, m_i-1\}\\
    p \mid s_j t_{j,1} \cdots t_{j,m_j-1} \implies p \mid s_i t_{i,1} \cdots t_{i,m_i-1} \quad \forall i,j \in \{0, \ldots, n\}
    \end{array}\right.\right\}.
\end{align*}

We can thus rewrite the relation \eqref{eq: step prim} in the shape
\begin{align}
\label{eq: step M_cgamma}
    \# N_{\bd}^* (B, \bone, \bone) =\sum_{(\bs, \bt) \in \calT_B} \varpi(\bs, \bt)  \sum_{\tilde \bv \in V_B(\bs,\bt)}M_{\bd, \bgamma}(B^k) 
\end{align}
due to the properties of the function $\varpi$ in Lemma~\ref{lem: varpi}.

The functions $M_{\bd, \bgamma}(B^k)$ can be estimated via \cite[Theorem 2.7]{BY21} whenever $\sum_{i=0}^{n-1}\frac 1{km_i(km_i+1)}\geq 1$. In the next section we improve the circle method result of Browning and Yamagishi to a version that applies under the weaker assumptions of Theorem \ref{thm: diagonal hypersurface}.

\section{Application of the circle method}\label{sec: circle method result}
Here we prove a sharper version of \cite[Theorem 2.7]{BY21} with parameters $N=0, H=1, \bh=0$.
Fix integers $2\le \tilde m_0 \le \dots \le  \tilde m_n$.
For $\bd\in\Z^{n+1}_{\neq0}$, $\bzeta\in\N^{n+1}$ and $\tilde B>0$, put 
\begin{align}\label{def: Mcgamma}
    \tilde M_{\bd, \bzeta}( \tilde B) = \# \left \{ \bu \in \N^{n+1} \left |     \zeta_i u_i^{\tilde m_i} \le \tilde B \ \forall i\in\{0,\dots,n\}, \, \sum_{i=0}^{n} d_i \zeta_i u_i^{\tilde m_i}=0 \right. \right\}. 
\end{align}

Thus, if $\tilde m_i = k m_i$ for $0 \le i \le n$ we have $M_{\bd,\bzeta}(B^k) = \tilde M_{\bd,\bzeta}(B^k)$.

Our next immediate goal is to establish an asymptotic formula for $\tilde M_{\bd, \bgamma}( \tilde B)$. Here, we use the Hardy--Littlewood circle method much in the way of \cite[Section 2]{BY21}, but we aim for a slightly sharper bound than the one presented in that paper. 

For any integer $m\geq 2$, let $s_0(m)$ be a positive integer such that for all $s \ge s_0(m)$ and $\epsilon>0$ one has
$$
    \int_0^1 \left| \sum_{1 \le x \le X} e(\alpha x^{m}) \right|^{2s} \mathrm{d} \alpha \ll X^{2s-m+\epsilon}. 
$$
Similarly, let $\sigma(m)$ be a positive integer with the property that for any $\beta \in [0,1)$ and any $q \in \N$ with $\|q \beta\| \le q^{-1}$  
one has the bound 
$$	
    \left| \sum_{1 \le x \le X} e(\beta x^{m}) \right| \ll X^{1+\epsilon} (X^{-1} + q^{-1} + qX^{-m})^{\sigma(m)}
$$
for all $\epsilon>0$.  
It follows from \cite[Theorem 2.1]{BP21} (see also \cite[Theorem 14.8]{Woo19} for the bound on $s_0(m)$) that the choices 
\begin{align*}
    s_0(m) = \min\{ 2^{m-1},  \textstyle{\frac12} m (m-1) + \lfloor \sqrt{2m+2}\rfloor \} \quad \text{ and } \quad \sigma(m)^{-1} = 2 s_0(m) 
\end{align*}
are admissible. 
(We remark here that even $\sigma(m)^{-1} = 2 s_0(m-1)$ is admissible. 
Since $s_0$ is strictly increasing, choosing $\sigma(m)$ as above constitutes a weakening.)

Suppose for the remainder of the section that the parameters $\bd\in\Z^{n+1}_{\neq0}$, $\bzeta\in\N^{n+1}$ and $\tilde B>0$ are fixed. For every $i\in\{0,\dots,n\}$, 
let $\tilde B_i=(\tilde B/\zeta_i)^{\frac 1{\tilde m_i}}$ and 
$$S_i(\alpha)=\sum_{1\leq u\leq \tilde B_i}e(\alpha d_i \zeta_i u^{\tilde m_i}),
$$ 
so that 
\[
    \tilde M_{\bd,\bzeta}(\tilde B) = \int_0^1\prod_{i=0}^n S_i(\alpha)\mathrm{d}\alpha.
\]
For fixed $\delta>0$, we recall the set of minor arcs from \cite[\S2]{BY21}, which is given by 
\begin{align*}
    \mathfrak m = [0,1]\smallsetminus \bigcup_{\substack{0\leq a\leq q\leq {\tilde B}^\delta\\ \gcd(a,q)=1}}\{\alpha\in[0,1):  |\alpha- a/q|<{\tilde B}^{\delta-1}\}.
\end{align*}

\begin{lemma}\label{lem:Weyl-bound}
	For $i\in\{0,\dots,n\}$, $0<\delta<\frac 1{(2n+5)\tilde m_n(\tilde m_n+1)}$ and $\epsilon > 0$, we have 
	$$
    	\sup_{\alpha \in \mathfrak m} |S_{i}(\alpha)| 
        \ll \tilde B^{1/\tilde m_i - \delta \sigma(\tilde m_i) + \epsilon} 
        \zeta_i^{-1/\tilde m_i + \sigma(\tilde m_i)}.
	$$
\end{lemma}
\begin{proof} 
	The proof is identical to that of \cite[Lemma 2.5]{BY21} (see in particular the last display on p.1083), with the only 
 difference in the precise meaning of the quantity $\sigma(\tilde m_i)$.
\end{proof}

Let
\begin{equation}\label{eq:Theta and Gamma}
	\tilde \Theta = \sum_{i=0}^{n} \frac{1}{2s_0(\tilde m_i)} - 1 \qquad \text{and} \qquad \tilde \Gamma = \sum_{i=0}^{n} \frac{1}{\tilde m_i} - 1.
 \end{equation}

\begin{lemma} 
\label{lem:minor arcs}
	Assume that 
	$
	   \tilde \Theta>0.
	$
	Let $0<\delta<\frac 1{(2n+5)\tilde m_n(\tilde m_n+1)}$ and $\epsilon>0$. Then 
	$$
	   \int_{\mathfrak m} \left| \prod_{i=0}^n S_i(\alpha)\right| \mathrm{d} \alpha \ll  \tilde B^{\tilde \Gamma - \delta \tilde \Theta + \epsilon}\prod_{i=0}^{n} \zeta_i^{1/(2s_0(\tilde m_i))-1/\tilde m_i}.
	$$
\end{lemma}

\begin{proof}
	From the condition $\tilde \Theta>0$ it follows that we can find $\beta_0, \ldots, \beta_n \in (0,1)$ so that 
	\begin{equation}\label{eq:minor_arcs}
		\sum_{i=0}^n \frac{\beta_i}{2s_0(\tilde m_i)}=1.
	\end{equation}
	Take $\ell_i =\frac{2s_0(\tilde m_i)}{\beta_i}$ for all $1\leq i\leq n$. It then follows from H\"older's inequality, the definition of $s_0(\tilde m_i)$ and Lemma \ref{lem:Weyl-bound} above that	
	\begin{align*}
        \int_{\mathfrak m} \left| \prod_{i=0}^n S_i(\alpha)\right| \mathrm{d} \alpha 
        &\ll \left(\prod_{i=0}^n \sup_{\alpha \in \mathfrak m} |S_{i}(\alpha)|^{1-\beta_i}\right) \int_0^1\prod_{i=0}^{n} |S_i(\alpha)|^{\beta_i}  \mathrm{d} \alpha \\
        &\ll   \prod_{i=0}^n \left(\sup_{\alpha \in \mathfrak m} |S_{i}(\alpha)|^{1-\beta_i} \left(\int_0^1 |S_i(\alpha)|^{\beta_i\ell_i} \mathrm{d} \alpha \right)^{1/\ell_i} \right)\\
        &\ll \tilde B^{\epsilon} \prod_{i=0}^n \left( \left(\tilde B^{1/\tilde m_i - \delta \sigma(\tilde m_i)} \zeta_i^{-1/\tilde m_i + \sigma(\tilde m_i)}\right)^{1-\beta_i}  \left(\tilde B/\zeta_i\right)^{\frac{\beta_i\ell_i - \tilde m_i}{\ell_i\tilde m_i}}\right)  \\
        &\ll \tilde B^{\omega_1+\epsilon} \prod_{i=0}^n \zeta_i^{\omega_{2,i}},
	\end{align*}
	where 
	$$
		\omega_1 = \sum_{i=0}^n \left((1/\tilde m_i - \delta \sigma(\tilde m_i))(1-\beta_i) + \frac{\beta_i\ell_i - \tilde m_i}{\ell_i\tilde m_i} \right)
	$$
	and 
	$$
		\omega_{2,i} = (-1/\tilde m_i + \sigma(\tilde m_i))(1-\beta_i) -\frac{\beta_i\ell_i - \tilde m_i}{\ell_i\tilde m_i} \qquad \forall i\in\{0,\dots,n\}. 
	$$
	After inserting our choice for $\ell_i$, a modicum of computation shows that 
	$$
		\omega_{2,i}
		= -\frac{1}{\tilde m_i} + \sigma(\tilde m_i) - \beta_i\left( \sigma(\tilde m_i) - \frac{1}{2 s_0(\tilde m_i)}  \right) = -\frac{1}{\tilde m_i} + \frac{1}{2s_0(\tilde m_i)}, 
	$$
	where in the last step we took advantage of our choice $(2 s_0(\tilde m_i))^{-1} = \sigma(\tilde m_i)$.
	Similarly, we have 
	$$
		\omega_1 = \sum_{i=0}^n \left(\frac{1}{\tilde m_i} - \frac{\beta_i}{2 s_0(\tilde m_i)} - \delta \sigma(\tilde m_i)(1-\beta_i)\right) = \tilde\Gamma - \delta \tilde\Theta,
	$$	
	where the last equality follows from \eqref{eq:minor_arcs}.
\end{proof}

Combining Lemma \ref{lem:minor arcs} with \cite[Lemma 2.2]{BY21} yields the desired asymptotic formula. 
Define 
\begin{align}\label{eq:sing series}
    \mathfrak S_{\bd,\bzeta}  =\sum_{q=1}^\infty \frac{1}{q^{n+1}}\sum_{\substack{a \mmod q \\ (a,q)=1}}\prod_{i=0}^n \sum_{r=1}^qe(a d_i \zeta_i r^{m_i}/q) \qquad \text{and} \qquad \mathfrak J_{\bd} = \int_{-\infty}^\infty \prod_{i=0}^n \left(\int_0^1 e(\lambda d_i \xi^{m_i}) d \xi\right) d \lambda,
\end{align}
noting that these coincide with the definitions of $\mathfrak S_{\bd;\bzeta}(\bh,H;N)$ and $\mathfrak J_{\bd}$ of \cite[Section 2.1]{BY21} with the choices $\bh = \bzero$, $H=1$ and $N=0$. Thus, we obtain the following.

\begin{theorem}
\label{thm: sharp BY}
    Let $2\le \tilde m_0 \le \dots \le  \tilde m_n$ be positive integers such that 
    \begin{align*}
        \sum_{i=0}^{n} \frac{1}{2 s_0(\tilde m_i)} > 1 \qquad \text{and} \qquad \sum_{i=0}^{n} \frac{1}{\tilde m_i} > 3.
    \end{align*}
    Let $\bd\in\Z^{n+1}_{\neq0}$, $\bzeta\in\N^{n+1}$, $\tilde B>0$.
    Let $0 < \delta < \frac{1}{\left( 2n + 5\right) \tilde m_n \left( \tilde m_n +1 \right)}$ and $\epsilon > 0$.
    With the notation introduced in \eqref{def: Mcgamma} and \eqref{eq:Theta and Gamma}, we have 
    \begin{align*}
        \tilde M_{\bd, \bzeta}(\tilde B) 
        &=\frac{\textfrak{S}_{\bd, \bzeta} \textfrak{J}_{\bd}}  {\prod_{i=0}^{n} \zeta_i^{\frac{1}{\tilde m_i}}} \tilde B^{\tilde \Gamma} 
        + O\left(E_1(\bzeta) + \frac{\tilde B^{\tilde \Gamma-\delta} E_2(\bzeta)}{ \prod_{i=0}^n \zeta_i^{\frac{1}{\tilde m_i}}}+ \tilde B^{\tilde \Gamma- \delta \tilde \Theta+\epsilon} \prod_{i=0}^n\zeta_i^{-\frac{1}{\tilde m_i} + \frac{1}{2 s_0(\tilde m_i)}}\right),
    \end{align*}

    and the error terms are given by
    \begin{align*}
        E_1(\bzeta) = \frac{\prod_{i=0}^n \tilde B_i}{\tilde B} \left( \frac{1}{\tilde B_0}+ \dots + \frac{1}{\tilde B_n} \right) \tilde B^{(2n+5)\delta} \qquad \text{and} \qquad
        E_2(\bzeta) = \sum_{q=1}^{\infty} q^{1-\Gamma + \epsilon} \prod_{i=0}^n \gcd(\zeta_i, q)^{\frac{1}{\tilde m_i}}.
    \end{align*}
\end{theorem}

\section{Proof of the main theorem}\label{sec:proof main theorem}

In this section we conclude the proof of Theorem \ref{thm: diagonal hypersurface}. We start with some numerical conditions that follow from the assumptions of Theorem \ref{thm: diagonal hypersurface}.

 \begin{lemma}
 \label{lem: conditions}
     The assumptions $k, m_0,\dots,m_n\geq 2$ and $\sum_{i=0}^{n} \frac{1}{2 s_0(k m_i)} > 1$ imply
     \begin{align*}
         \frac{1}{2s_0(km_i)} \le \frac{1}{km_i} - \frac{1}{k(m_i+1)} \quad \forall i\in\{0,\dots,n\} \qquad \text{and} \qquad \sum_{i=0}^{n} \frac{1}{k m_i} > 3.
     \end{align*}
 \end{lemma}
 \begin{proof}

    The first inequality in the statement can be rearranged to
$$
    km_i(m_i+1) \le 2 s_0(km_i).
$$
Now, suppose first that $km_i \ge 6$, so that $s_0(km_i) = \frac12 km_i(km_i-1) + \lfloor \sqrt{2km_i+2}\rfloor$. In that situation, the above bound becomes
$$
    km_i(m_i+1) \le km_i(km_i-1) + 2\lfloor \sqrt{2km_i+2}\rfloor, 
$$
which can be rearranged to 
$$
    km_i((k-1)m_i-2) + 2\lfloor \sqrt{2km_i+2}\rfloor \ge 0,
$$
which is clearly satisfied for $k \ge 2$ and $m_i \ge 2$. This settles all cases in which $k\ge 3$ or $m_i \ge 3$. The remaining case $k=m_i=2$ can be checked by hand.

Finally, the second statement follows upon observing that for $i\in \{0,\dots,n\}$, $2s_0(km_i) \ge 3 km_i$, as $km_i\geq 4$.
 \end{proof}

Now we continue the proof of Theorem \ref{thm: diagonal hypersurface} from where we left off at the end of Section \ref{sec: the counting problem}.

 \begin{proposition} \label{proposition}
    Let $\bd \in \Z_{\neq0}^{n+1}$. Fix integers $m_0, \ldots, m_n \ge 2$ and $k \ge 2$, and put $\Gamma = \sum_{i=0}^n \frac{1}{km_i} -1$. There exists a real number $\eta>0$ such that 
    \begin{align}
        \# N_{\bd}^* ( B, \bone, \bone ) = C_{\bd}B^{k\Gamma} + O(B^{k\Gamma - \eta}),
    \end{align}
    where 
    \begin{align}\label{eq: leading constant}
        C_{\bd}=\textfrak{J}_{\bd} \sum_{(\bs, \bt) \in \calT_\infty} \varpi(\bs, \bt) \sum_{\tilde \bv \in V_\infty(\bs,\bt)} \textfrak{S}_{\bd, \bgamma} \prod_{i=0}^{n} \gamma_i^{-\frac{1}{k m_i} }
    \end{align}
    with $\mathfrak S_{\bd, \bgamma}$ and $\mathfrak J_{\bd}$ as in \eqref{eq:sing series}, and $\bgamma$ defined via \eqref{eq: gammaj}. 
 \end{proposition}

\begin{proof}
Our strategy is to apply Theorem \ref{thm: sharp BY}, with $\tilde B=B^k$ and $\tilde m_i=km_i$ for $0\leq i\leq n$, to estimate the sets $M_{\bd, \bgamma}(B^k)$  in \eqref{eq: step M_cgamma}. Note that our choice of the parameters ensures that $M_{\bd,\bgamma}(B^k) = \tilde M_{\bd,\bgamma}(\tilde B)$, and that the quantity $\tilde \Gamma$ coincides with the $\Gamma$ defined in the statement of the proposition.
Thus, for every $\epsilon > 0$ and $0<\delta<\frac 1{(2n+5)km_n(km_n+1)}$ we obtain
\begin{align}
\label{eq: step circle method}
\begin{split}
    \# N_{\bd}^* ( B, \bone, \bone )
    &= \sum_{(\bs,\bt) \in \calT_B}\varpi(\bs, \bt)  \sum_{\tilde \bv \in V_B(\bs,\bt)} \left [ \frac{\textfrak{S}_{\bd, \bgamma} \textfrak{J}_{\bd}}{\prod_{i=0}^{n} \gamma_i^{\frac{1}{k m_i}}} B^{k\Gamma} 
    + O\left(E_1(\bgamma) + \frac{B^{k\Gamma-k\delta} E_2(\bgamma)}{\prod_{i=0}^n \gamma_i^{\frac{1}{k m_i}}}\right)\right.\\ 
    &\qquad +\left.O \left( B^{k\Gamma- k\delta\Theta+\epsilon} \prod_{i=0}^n\gamma_i^{-\frac{1}{k m_i}+\frac 1{2s_0(km_i)}}\right) \right],
\end{split}
\end{align}
where $\Theta = \sum_{i=0}^n\frac{1}{2s_0(km_i)}-1$. 
We observe that by \cite[p.~1093]{BY21}, the leading constant is 
\begin{align}\label{eq: main term}
    \sum_{(\bs,\bt) \in \calT_B} \varpi(\bs, \bt) \sum_{\tilde \bv \in V_B(\bs,\bt)} \textfrak{S}_{\bd, \bgamma}\textfrak{J}_{\bd} \prod_{i=0}^{n} \gamma_i^{-\frac{1}{k m_i} }=C_{\bd} + O(B^{-\eta}),
\end{align}
where $C_{\bd}$ is the expression given in \eqref{eq: leading constant} and $\eta$ is some suitable positive number.

It thus remains to bound the error terms. Put 
\begin{align*}
    F_1(B)=B^{-k\Gamma}\sum_{(\bs,\bt) \in \calT_B}\varpi(\bs, \bt) \sum_{\tilde \bv \in V_B(\bs,\bt)}E_1(\bgamma), \qquad 
    F_2(B)=\sum_{(\bs,\bt) \in \calT_B}\varpi(\bs, \bt) \sum_{\tilde \bv \in V_B(\bs,\bt)}
    \frac{B^{-k\delta}E_2(\bgamma)}{\gamma_i^{\frac{1}{km_i}}}
\end{align*}
and
\begin{align*}
    F_3(B)=B^{-k\delta\Theta +\epsilon} \sum_{(\bs,\bt) \in \calT_B} \varpi(\bs, \bt)\sum_{\tilde \bv \in V_B(\bs,\bt)} \prod_{i=0}^n \gamma_i^{-\frac{1}{km_i}+\frac 1{2s_0(km_i)}},
\end{align*}
then the desired result will follow if we can show that $F_j(B) \ll B^{-\eta}$ for some $\eta>0$.

We now begin with the estimation of $F_1(B)$. Upon inserting the definition of $E_1(B)$ from the statement of Theorem \ref{thm: sharp BY}, one can show by a simple computation (see also \cite[(3.6)]{BY21}), using the last statement of Lemma \ref{lem: varpi}, that 
\begin{align*}
    F_1(B)=B^{k(2n+5)\delta+\epsilon} \sum_{l=0}^n B^{-\frac{1}{m_l}} \sum_{(\bs,\bt) \in \calT_B} \sum_{\tilde \bv \in V_B(\bs,\bt)} \prod_{\substack{i=1 \\i \neq l}}^n\gamma_i^{-\frac{1}{km_i}}. 
\end{align*}
For $0\leq i\leq n$, we introduce the  notation
\begin{align*} 
    w_i=\tilde v_{i,1}^{m_i+1} \cdots \tilde v_{i, m_i-1}^{2m_i-1} \quad \textup{and} \quad \tau_i= s_i^{m_i} \prod_{r=1}^{m_i-1} t_{i,r}^{m_i+r},
\end{align*}
so that $\gamma_i=w_i^k \tau_i^k$. In that notation we can write 
\begin{align*}
    F_1(B) &= B^{k(2n+5)\delta+\epsilon} \sum_{l=0}^n B^{-\frac{1}{m_l}} \sum_{(\bs,\bt) \in \calT_B} \sum_{\tilde \bv \in V_B(\bs,\bt)}\prod_{\substack{i=1 \\i \neq l}}^n (w_i\tau_i)^{-\frac{1}{m_i}}  \nonumber\\ 
    &\ll B^{k(2n+5)\delta+\epsilon} \sum_{l=0}^n B^{-\frac{1}{m_l}} \sum_{(\bs,\bt) \in \calT_B} \left( \sum_{\tilde  v_{l,1}, \dots, \tilde v_{l,m_l-1} : w_l \le \frac{B}{\tau_l}} 1 \right) \prod_{\substack{i=1 \\i \neq l}}^n \tau_i^{-\frac{1}{
    m_i}}  \sum_{\tilde v_{i,1}, \dots, \tilde v_{i,m_i-1} : w_i \le \frac{B}{\tau_i}}  \left( \frac{1}{w_i} \right)^{\frac{1}{m_i}}. 
\end{align*}
Using the estimates
\begin{align*}
    \sum_{v_1^{m+1} \cdots v_{m-1}^{2m-1} \le \frac{B}{\tau}} 1 \ll \sum_{v_2,\dots ,v_m=1}^{\infty} \left( \frac{B/\tau}{v_2^{m+2} \cdots v_{m-1}^{2m-1}} \right)^{\frac{1}{m+1}} \ll \left( \frac{B}{\tau}\right)^{\frac{1}{m+1}}
\end{align*}
and
\begin{align*}
    \sum_{v_1^{m+1} \cdots v_{m-1}^{2m-1} \le \frac{B}{\tau}} \left( \frac{1}{v_1^{m+1} \cdots v_{m-1}^{2m-1}} \right)^{\frac{1}{m}} \ll 1
\end{align*}
from \cite[p.~1090]{BY21} within the above bound, it follows that
\begin{align*}
    F_1(B) 
    &\ll  B^{k(2n+5)\delta+\epsilon} \sum_{l=0}^n B^{-\frac{1}{m_l}} \sum_{(\bs,\bt) \in \calT_B} \left(\frac{B}{\tau_l}\right)^{\frac{1}{m_l+1}} \prod_{\substack{i=1 \\i \neq l}}^n  \tau_i^{-\frac{1}{m_i}}\\
    &\ll B^{-\frac{1}{m_n(m_n+1)} +k(2n+5)\delta+\epsilon} \sum_{(\bs,\bt) \in \calT_B}  \prod_{i=0}^n \tau_i^{-\frac{1}{(m_i+1)}}.
\end{align*}
The sum can be bounded by a slight modification of \cite[(3.8)]{BY21}. 
Write $\sigma_i = s_i\prod_{r=1}^{m_i-1} t_{i,r}$,  
then clearly $\sigma_i^m \le \tau_i$, and thus
\begin{align*}
    \sum_{(\bs,\bt) \in \calT_B}  \prod_{i=0}^n \tau_i^{-\frac{1}{(m_i+1)}}
    & \le\sum_{(\bs,\bt) \in \calT_B}\prod_{i=0}^{n} \sigma_i^{-\frac{m_i}{m_i+1}}
    =\prod_p \left( 1 + \sum_{\substack{(\bs,\bt) \in \calT_B \\ (\bs,\bt) \neq (\bone,\bone)}} p^{-\sum_{i=0}^n \val_p(\sigma_i) \frac{m_i}{m_i+1}}\right)\\
    & \le \prod_p \left( 1 + \prod_{i=0}^n p^{-\frac{m_i}{m_i+1}} \sum_{\substack{(\bs, \bt) \in \calT_B \\\val_p(\sigma_i)\ge 1} } 1 \right)
    \le \prod_p \left( 1 + \prod_{i=0}^n p^{-\frac{m_i}{m_i+1}} (2m_i-1) \right) \ll 1,
\end{align*}
as $\sum_{i=0}^n \frac{m_i}{m_i+1} > 1$ under the assumption $m_0,\dots,m_n\geq 2$. 
Thus altogether we obtain the bound 
\begin{align}\label{eq: F1 bound}
    F_1(B)\ll B^{-\frac{1}{m_n(m_n+1)} +k(2n+5)\delta+\epsilon},
\end{align}
which is sufficient for our purposes, provided that $\delta$ was taken small enough.

We now turn to $F_2(B)$. As before, we rewrite the quantity under consideration by inserting the definition of $E_2(B)$ (see also \cite[(3.6)]{BY21}), whereupon we can use the upper bound  
\begin{align*}
    F_2(B) &=  B^{-k\delta+\epsilon}\sum_{(\bs,\bt) \in \calT_B} \sum_{\tilde \bv \in V_B(\bs,\bt)} \sum_{q=1}^\infty q^{1-\Gamma+\varepsilon} \prod_{i=0}^n \frac{\gcd(\gamma_i, q)^{\frac{1}{km_i}}}{\gamma_i^{\frac{1}{km_i}}}\\
    &\le B^{-k\delta+\epsilon} \sum_{q=1}^{\infty} q^{1-\Gamma+\varepsilon} f_1(q)f_2(q),
\end{align*}
with 
\begin{align*}
    f_1(q)=\sum_{(\bs,\bt) \in \calT_\infty} \prod_{i=0}^n \left( \frac{\gcd(\tau_i^k, q)}{\tau_i^k}\right)^{\frac{1}{km_i}} \qquad \text{and} \qquad
    f_2(q)= \sum_{\tilde \bv \in V_B(\bone,\bone)}
    \prod_{i=0}^n \left(\frac{\gcd ( w_i^k, q)}{{w_i^k}}\right)^{\frac{1}{km_i}}.
\end{align*}
Since $\tilde v_{i,1},\dots,\tilde v_{i,m_i-1}$ are pairwise coprime for $\tilde \bv \in V_B(\bone,\bone)$, we can estimate the second expression via 
\begin{align*}
    f_2(q)
    \leq \prod_{i=0}^n\prod_{r=1}^{m_i-1} 
    \sum_{\tilde v_{i,r}\leq B^{\frac 1{m_i+r}}}\mu^2(\tilde v_{i,r})
    \frac{\gcd(\tilde v_{i,r}^{km_i+kr},q)^{\frac{1}{km_i}}}{\tilde v_{i,r}^{(m_i+r)/{m_i}}}
    \ll q^\epsilon
\end{align*}
for all $\epsilon>0$, where the last bound follows from \cite[(3.9)]{BY21}.

For $f_1(q)$ we proceed as in \cite[p.~1091]{BY21}.
Let 
$$
    \mathscr T = \left\{ \overline{\mathbf{\tau}}=(\overline\tau_0,\dots,\overline\tau_n)\in\N^{n+1}: \begin{array}{c}
    \val_p(\overline\tau_i) \in \{0,m_i,m_i+1,\dots,3m_i-1\} \quad \forall p, \forall i\in\{0,\dots,n\} \\ p | \overline\tau_i \implies p | \overline\tau_j \quad \forall i,j \in \{0,\dots,n\} \end{array}\right\}.
$$ 
For every $\overline{\mathbf{\tau}}\in\mathscr T$ there is a unique pair $(\bs,\bt)\in\calT_\infty$ such that $\overline \tau_i=s_i^{m_i}\prod_{r=1}^{m_i-1}t_{i,r}^{m_i+r}$ for all $i\in\{0,\dots,n\}$. Then
\begin{align*}
    f_1(q)\leq\sum_{\overline{\mathbf{\tau}}\in\mathscr T}\prod_{i=0}^n \left( \frac{\gcd(\overline\tau_i^k, q)}{\overline\tau_i^k}\right)^{\frac{1}{km_i}}\leq\prod_p\left(1+\prod_{i=0}^n\sum_{m_i\leq\alpha_i\leq 3m_i-1}p^{(\min\{k\alpha_i,\val_p(q)\}-k\alpha_i)/km_i}\right)=O(1),
\end{align*}
as the local contribution to the product is $1+O(p^{-(n+1)})$ if $p\nmid q$, and $O(1)$ if $p\mid q$.
Thus 
\begin{align}\label{eq: F2 bound}
    F_2(B)\ll B^{-k\delta+\epsilon} \sum_{q=1}^{\infty} q^{1-\Gamma+\epsilon}\ll B^{-k\delta+\epsilon}
\end{align}
for all $\epsilon>0$, as $\sum_{i=0}^n\frac1{{km_i}}> 3$ by Lemma \ref{lem: conditions}.

Lastly, we turn our attention to $F_3(B)$. By Lemma \ref{lem: conditions} we have
\begin{align*}
    &\sum_{(\bs,\bt) \in \calT_B} \sum_{\tilde \bv \in V_B(\bs,\bt)} \prod_{i=0}^n \gamma_i^{-\frac{1}{km_i}+\frac 1{2s_0(km_i)}}
    \leq\sum_{(\bs,\bt) \in \calT_B} \left(\sum_{\tilde \bv \in V_B(\bs,\bt)} \prod_{i=0}^n w_i^{-\frac{1}{m_i+1}}\right)\prod_{i=0}^n\tau_i^{-\frac{1}{m_i+1}}\\
    &\ll \left( \prod_{i=0}^n \sum_{\tilde v_{i,1}^{m_i+1}\cdots \tilde v_{i,m_i-1}^{2m_i-1}\leq B}  \left( \tilde v_{i,1}^{m_i+1}\cdots \tilde v_{i,m_i-1}^{2m_i-1} \right)^{-\frac{1}{m_i+1}} \right)\sum_{(\bs,\bt) \in \calT_B} \prod_{i=0}^n\tau_i^{-\frac{1}{m_i+1}}
    \ll B^\epsilon
\end{align*}
for all $\epsilon>0$, where the last estimate follows from \cite[(3.8)]{BY21} and the bound $\ll\log B$ for the sum over $v_{i,r}$ in \cite[p.~1092]{BY21}.
Thus 
\begin{align}\label{eq: F3 bound}
    F_3(B)\ll B^{-k\delta\Theta+\epsilon}
\end{align}
for all $\epsilon>0$.

The desired conclusion follows now upon combining \eqref{eq: step circle method} with \eqref{eq: main term} as well as the bounds \eqref{eq: F1 bound}, \eqref{eq: F2 bound} and \eqref{eq: F3 bound}. 
\end{proof}

Theorem \ref{thm: diagonal hypersurface} is now immediate upon combining 
Lemma \ref{lem: set of campana points}, \eqref{eq: step 1/2},
\eqref{eq: step odd} and Proposition \ref{proposition}. In particular, the leading constant in \eqref{eq:asymptotic} is
    \begin{equation}\label{eq: leading constant thm 1.1}
    C=\begin{cases}
    \frac 12 \sum_{\beps\in\{\pm1\}^{n+1}}C_{\beps\bc} & \text{ if $k$ is odd},\\
    2^n C_{\bc} & \text{ if $k$ is even}.
    \end{cases}
    \end{equation}

\begin{remark}
    Note that the condition that $k \ge 2$ enters in two places. On the one hand, we need $\Gamma > 3$ to control the main term in Theorem \ref{thm: sharp BY}, and on the other hand it plays a role in the estimation of $F_3(B)$, via the application of Lemma \ref{lem: conditions}. By implementing a suitable pruning argument in the treatment of the minor arcs in the proof of Theorem \ref{thm: sharp BY}, it seems likely that the last error term in that theorem can be improved, giving rise to a corresponding quantity $F_3$ with better convergence properties. As for the main term, the condition $\Gamma > 3$ can presumably be relaxed by means of the techniques described in \cite[Chapter 4]{vaughan_1997}. Consequently, the expectation is that extending Proposition \ref{proposition}, and thus Theorem \ref{thm: diagonal hypersurface}, for $k=1$ should be a question of determination rather than any potential structural obstacles.
\end{remark}

\bibliographystyle{alpha}
\bibliography{proceedings}
\end{document}